\documentclass[11pt]{amsart}

\usepackage[utf8]{inputenc}
\usepackage[T1]{fontenc}
\usepackage{amsmath,amssymb,amsthm}
\usepackage{enumitem}
\usepackage{booktabs}
\usepackage{tabularx}
\usepackage[margin=1in]{geometry}
\usepackage[colorlinks=true,linkcolor=blue,citecolor=blue,urlcolor=blue]{hyperref}

\newtheorem{theorem}{Theorem}[section]
\newtheorem{lemma}[theorem]{Lemma}
\newtheorem{proposition}[theorem]{Proposition}
\newtheorem{corollary}[theorem]{Corollary}
\theoremstyle{definition}
\newtheorem{definition}[theorem]{Definition}
\newtheorem{example}[theorem]{Example}
\newtheorem{problem}[theorem]{Problem}
\theoremstyle{remark}
\newtheorem{remark}[theorem]{Remark}

\numberwithin{equation}{section}

\newcommand{\Z}{\mathbb{Z}}
\newcommand{\Q}{\mathbb{Q}}
\newcommand{\R}{\mathbb{R}}
\newcommand{\C}{\mathbb{C}}
\newcommand{\F}{\mathbb{F}}
\newcommand{\fl}[1]{\left\lfloor #1 \right\rfloor}
\newcommand{\fr}[1]{\left\{ #1 \right\}}
\newcommand{\Gal}{\operatorname{Gal}}
\newcommand{\Frob}{\operatorname{Frob}}

\makeatletter
\@namedef{subjclassname@2020}{\textup{2020} Mathematics Subject Classification}
\makeatother

\title{A Note on Polynomial Sequences Modulo Integers \\ (revised version)}
\author{Mohammad Javaheri}
\address{Department of Mathematics and Statistics, Siena University\\
515 Loudon Road, Loudonville, NY 12211, USA}
\email{mjavaheri@siena.edu}


\begin{document}

\begin{abstract}
For a real polynomial $P(x)$ we study the integer sequence $\lambda(P)=(\fl{P(k)})_{k\ge1}$ modulo integers. We show that for nonlinear $P$ the sequence $\lambda(P)$ is uniformly distributed (u.d.) in $\Z$ if and only if $P$ has an irrational coefficient other than the constant term. In fact, if all non-constant coefficients of a nonlinear $P$ are rational, then $\lambda(P)$ fails to be u.d.\ modulo every prime in a set of primes of positive density. We also study completeness: $P$ is complete modulo $m$ if $\fl{P(x)}$ meets every residue class modulo $m$. For $P$ of even degree we prove that $\lambda(P)$ meets every residue class modulo every integer if and only if $P$ has an irrational non-constant coefficient. The same holds for every polynomial $P(x)=ax^n+c$ with $n\ge2$, and for every nonlinear $P$ with $P(x)-P(0)\in\Z[x]$. For linear $P$, both properties are characterized completely, and the two notions differ.
\end{abstract}

\maketitle

\section{Introduction}

Weyl's theorem~\cite{Wey} says that the fractional parts of $P(k)$ are equidistributed when $P$ has an irrational non-constant coefficient. This paper asks the converse, arithmetic question: when do the integers $\fl{P(k)}$ distribute evenly, or at least reach every residue class, modulo every integer $m$?

The question sits between two classical topics. Uniform distribution of integer sequences in residue classes was introduced by Niven~\cite{Niv} and developed by Uchiyama~\cite{Uch}, Narkiewicz and others; see~\cite{Nar} and \cite[Ch.~5]{KN}. Niven proved that a polynomial with integer coefficients is u.d.\ in $\Z$ only if it is $\pm x+c$. For integer polynomials, uniform distribution modulo a prime $p$ is the same as being a permutation polynomial modulo $p$, a subject with a long history going back to Hermite and Dickson \cite[Ch.~7]{LN} and to Schur's conjecture, proved by Fried~\cite{Fri}. Our results extend this picture from $\Z[x]$ to arbitrary real polynomials, where the floor function enters.

\medskip\noindent\textbf{Notation.} For a real number $r$, $\fl{r}$ is the greatest integer $\le r$ and $\fr{r}=r-\fl{r}$. Throughout, $P(x)=a_nx^n+\dots+a_1x+a_0$ is a real polynomial of degree $n$. We write $\lambda(P)$ for the sequence $(\fl{P(k)})_{k\ge1}$. We say that $P$ is \emph{rational up to its constant term} if $a_1,\dots,a_n\in\Q$.

\subsection{Uniform distribution}

A sequence $(r_k)_{k\ge1}$ of real numbers is \emph{u.d.\ mod $1$} if, for all $0\le a<b\le1$,
\[
\lim_{N\to\infty}\frac1N\,\#\{1\le k\le N: a\le\fr{r_k}<b\}=b-a.
\]
An integer sequence $(a_k)_{k\ge1}$ is \emph{u.d.\ mod $m$} ($m\ge2$) if, for every integer $i$,
\begin{equation}\label{eq:ud}
\lim_{N\to\infty}\frac1N\,\#\{1\le k\le N: a_k\equiv i\pmod m\}=\frac1m.
\end{equation}
The sequence is \emph{u.d.\ in $\Z$} if it is u.d.\ mod $m$ for every $m\ge2$. If it is u.d.\ mod $m$, it is u.d.\ mod every divisor of $m$ \cite[Thm.~5.1]{Niv}. Since every $m$ divides arbitrarily large integers, u.d.\ in $\Z$ is equivalent to u.d.\ mod $m$ for all sufficiently large $m$.

The link with u.d.\ mod $1$ is simple. For integers $i$ and $m\ge2$ we have $\fl{r}\equiv i\pmod m$ exactly when $\fr{r/m}$ lies in an interval of length $1/m$ determined by $i$. Hence, if $(r_k/m)$ is u.d.\ mod $1$ for every $m\ge2$, then $(\fl{r_k})$ is u.d.\ in $\Z$ \cite[Ch.~5]{KN}. This yields many examples.

\begin{example}[Weyl]\label{ex:weyl}
If $P$ has an irrational coefficient other than $a_0$, then so does $P/m$ for every $m$, and Weyl's theorem~\cite{Wey} (see \cite[Ch.~1, Thm.~3.2]{KN}) shows that $(P(k)/m)$ is u.d.\ mod $1$. Hence $\lambda(P)$ is u.d.\ in $\Z$. Weyl's theorem grew out of the work of Hardy and Littlewood on the fractional parts of $k^n\theta$~\cite{HL}. The converse of Example~\ref{ex:weyl} for nonlinear $P$ is Theorem~\ref{thm:nonlin-elem} below.
\end{example}

\begin{example}[Non-integer powers]
Let $f(x)=\beta x^\alpha$ with $\beta\ne0$ and $\alpha>0$ not an integer. Then $(f(k)/m)$ is u.d.\ mod $1$ for every $m$: for $0<\alpha<1$ this is Fej\'er's theorem, and for $\alpha>1$ it follows from van der Corput's difference theorem \cite[Ch.~1, \S\S2--3]{KN}. Hence $(\fl{f(k)})$ is u.d.\ in $\Z$.
\end{example}

\begin{example}[Linear integer polynomials]\label{ex:niven}
If $P(x)=\pm x+c$ with $c\in\R$, then $\fl{P(k)}=\pm k+\fl{c}$, which is u.d.\ in $\Z$. Conversely, if $P(x)-P(0)\in\Z[x]$, then $\fl{P(k)}=P(k)-P(0)+\fl{P(0)}$ is an integer polynomial, so by Niven's theorem~\cite{Niv} $\lambda(P)$ is u.d.\ in $\Z$ only if $P(x)=\pm x+c$.
\end{example}

\begin{example}[Exponential sequences]
Fix $\beta\ne0$. By Koksma's metric theorem~\cite{Kok} (see \cite[Ch.~1, Thm.~4.3]{KN} and~\cite{LeV}), for each $m\ge2$ the sequence $(\beta\alpha^k/m)$ is u.d.\ mod $1$ for almost every $\alpha>1$. Taking the countable intersection over $m$, the sequence $(\fl{\beta\alpha^k})$ is u.d.\ in $\Z$ for almost every $\alpha>1$.
\end{example}

\subsection{Main results on uniform distribution}

Our first result settles which polynomial sequences $\lambda(P)$ are u.d.\ in $\Z$.

\begin{theorem}\label{thm:main-ud}
Let $P$ be a real polynomial. Then $\lambda(P)$ is u.d.\ in $\Z$ if and only if either $P$ has an irrational coefficient other than its constant term, or $P(x)=x/l+c$ for some nonzero integer $l$ and some $c\in\R$.
\end{theorem}

The linear case is Theorem~\ref{thm:linear-ud}, which extends Niven's theorem on $(\fl{\alpha k})$ \cite[Thm.~3.1]{Niv} to $(\fl{\alpha k+\beta})$. The nonlinear case is Theorem~\ref{thm:nonlin-elem} (an elementary argument) together with Theorem~\ref{thm:nonlin-cheb} (a stronger, Chebotarev-based argument).

\subsection{Completeness}

A weaker property than uniform distribution is that the sequence reaches every residue class.

\begin{definition}
A real polynomial $P$ is \emph{complete modulo $m$} if for every integer $r$ there is $x\in\Z$ with $\fl{P(x)}\equiv r\pmod m$. It is \emph{complete in $\Z$} if it is complete modulo every $m\ge2$ (equivalently, modulo all sufficiently large $m$).
\end{definition}

For $P$ rational up to its constant term, it makes no difference whether $x$ ranges over $\Z$ or over the positive integers, because $\fl{P(x)}$ is periodic modulo $m$ (Lemma~\ref{lem:periodic}). Completeness is a statement about value sets. For $P\in\Z[x]$ and a prime $p$, completeness modulo $p$ means that $P$ is surjective on $\F_p$, hence a permutation polynomial. The value sets of polynomials over finite fields are well studied~\cite{BSD}, \cite[Ch.~7]{LN}. Floors bring in a new feature: $\fl{P(x)}$ modulo $p$ is the union of the value sets of several shifted integer polynomials (equation~\eqref{eq:shift}).

By Example~\ref{ex:weyl}, an irrational non-constant coefficient makes $P$ complete in $\Z$. For linear $P$ the converse fails, since every $P(x)=\alpha x+\beta$ with $0<|\alpha|\le1$ is complete in $\Z$ (Theorem~\ref{thm:linear-complete}). For even degree, the converse holds.

\begin{theorem}\label{thm:even}
Let $P$ be a real polynomial of even degree $n\ge2$. The following are equivalent.
\begin{enumerate}[label=\textup{(\arabic*)}]
\item $P$ is complete modulo all sufficiently large primes.
\item $P$ has an irrational coefficient other than its constant term.
\item $\lambda(P)$ is u.d.\ in $\Z$.
\item $P$ is complete in $\Z$.
\end{enumerate}
\end{theorem}

We also prove the same equivalence for two families of odd degree: polynomials with $P(x)-P(0)\in\Z[x]$ (Proposition~\ref{prop:integer}), and polynomials $ax^n+c$ of any degree $n\ge2$ (Theorem~\ref{thm:monomial}). Section~\ref{sec:open} states the open problem that remains for general polynomials of odd degree.

Our main tool for completeness is the Chebotarev density theorem in the form of Lemma~\ref{lem:cheb}. It converts a question about polynomial roots modulo $p$ into a question about a Galois group acting without fixed points on roots. This links our problem to the study of polynomials with roots modulo every prime~\cite{BB,BBH}, and to Jordan's theorem on derangements in transitive groups~\cite{Ser}.

\section{Preliminaries}\label{sec:prelim}

Throughout this section $P(x)=a_nx^n+\dots+a_0$ is rational up to its constant term. Let $N\ge1$ be the least common multiple of the denominators of $a_1,\dots,a_n$, so that $Na_i\in\Z$ for $1\le i\le n$. Put
\[
Q(x)=N\bigl(P(x)-P(0)\bigr)\in\Z[x],\qquad A_j=N\fl{P(j)}-Q(j)\in\Z\quad(0\le j<N).
\]
Note that $Q$ has the same degree $n$ as $P$ (for $n\ge1$).

\begin{lemma}\label{lem:periodic}
\begin{enumerate}[label=\textup{(\alph*)}]
\item For all integers $x$ and $t$, $P(x+Nt)-P(x)$ is an integer divisible by $t$. Consequently, for every $m\ge2$ the sequence $\fl{P(x)}\bmod m$, $x\in\Z$, is periodic with period $Nm$.
\item For all integers $k$ and $0\le j<N$,
\begin{equation}\label{eq:shift}
N\fl{P(Nk+j)}=Q(Nk+j)+A_j.
\end{equation}
\item Let $p$ be a prime with $p\nmid N$. Then the set of residues of $N\fl{P(x)}$ modulo $p$, as $x$ runs over $\Z$, equals
\[
V_p=\bigcup_{j=0}^{N-1}\{\,Q(y)+A_j: y\in\F_p\,\}.
\]
In particular, $P$ is complete modulo $p$ if and only if $V_p=\F_p$.
\end{enumerate}
\end{lemma}

\begin{proof}
(a) We have $P(x+Nt)-P(x)=\sum_{i=1}^n a_i\bigl((x+Nt)^i-x^i\bigr)$. Each difference $(x+Nt)^i-x^i$ equals $Nt$ times an integer, and $Na_i\in\Z$. So every term is an integer multiple of $t$. Since $P(x+Nm)-P(x)$ is an integer, $\fl{P(x+Nm)}=\fl{P(x)}+\bigl(P(x+Nm)-P(x)\bigr)\equiv\fl{P(x)}\pmod m$.

(b) By (a), $P(Nk+j)-P(j)$ is an integer. Hence $\fl{P(Nk+j)}=\fl{P(j)}+\bigl(Q(Nk+j)-Q(j)\bigr)/N$. Multiply by $N$.

(c) Fix $j$. Since $p\nmid N$, as $k$ runs over $\Z$ the integer $Nk+j$ runs over every residue class modulo $p$. Hence, by~\eqref{eq:shift}, the residues of $N\fl{P(x)}$ with $x\equiv j\pmod N$ are exactly $Q(\F_p)+A_j$. Since multiplication by $N$ is a bijection of $\F_p$, $P$ is complete modulo $p$ if and only if $V_p=\F_p$.
\end{proof}

By periodicity, $P$ is complete modulo $m$ if and only if each residue class modulo $m$ occurs at least once in a period. It is u.d.\ mod $m$ if and only if each class occurs exactly $N$ times in a period of length $Nm$. This also shows that letting $x$ range over $\Z$ or over the positive integers makes no difference. We record the case of a prime modulus.

\begin{lemma}\label{lem:count}
Let $p\nmid N$ be prime. Then $\lambda(P)$ is u.d.\ mod $p$ if and only if, for every $c\in\F_p$,
\begin{equation}\label{eq:count}
\#\{(y,j)\in\F_p\times\{0,\dots,N-1\}: Q(y)+A_j\equiv c\pmod p\}=N.
\end{equation}
\end{lemma}

\begin{proof}
Write the period $\{0,\dots,Np-1\}$ as $\{Nt+j: 0\le t<p,\ 0\le j<N\}$. For fixed $j$, the residues of $y=Nt+j$ modulo $p$ run over all of $\F_p$ exactly once, since $p\nmid N$. By~\eqref{eq:shift}, the number of $k$ in the period with $N\fl{P(k)}\equiv c$ is the left side of~\eqref{eq:count}. Uniform distribution of $\lambda(P)$ mod $p$ is equivalent to uniform distribution of $N\lambda(P)$ mod $p$, that is, to every count being $Np/p=N$.
\end{proof}

We use the Chebotarev density theorem in the following standard form; see \cite[Ch.~I, \S8 and Ch.~VII, \S13]{Neu} and~\cite{LS}.

\begin{lemma}\label{lem:cheb}
Let $F\in\Z[x]$ be nonconstant, let $L$ be its splitting field, let $\Omega\subset\C$ be its set of roots, and let $G=\Gal(L/\Q)$, acting on $\Omega$.
\begin{enumerate}[label=\textup{(\alph*)}]
\item Let $p$ be a prime not dividing the leading coefficient of $F$ nor the discriminant of the product of the distinct irreducible factors of $F$. Then $p$ is unramified in $L$, and the number of distinct roots of $F$ in $\F_p$ equals the number of points of $\Omega$ fixed by the Frobenius element $\Frob_{\mathfrak P}$ at any prime $\mathfrak P$ of $L$ above $p$.
\item For each conjugacy class $C$ of $G$, the set of primes $p$ whose Frobenius class is $C$ has natural density $|C|/|G|>0$.
\end{enumerate}
Consequently: if some $\sigma\in G$ fixes no point of $\Omega$, then $F$ has no root modulo $p$ for a set of primes $p$ of positive density \textup{(}conjugates of $\sigma$ also fix no point\textup{)}. And if $F$ is squarefree, then $F$ splits into $\deg F$ distinct linear factors modulo $p$ for a set of primes of density $1/|G|$ \textup{(}take $\sigma=1$\textup{)}.
\end{lemma}

\section{Uniformly distributed polynomial sequences}

\subsection{The linear case}

Niven \cite[Thm.~3.1]{Niv} proved that $(\fl{\alpha k})$ is u.d.\ in $\Z$ if and only if $\alpha$ is irrational or $\alpha=1/l$ for a nonzero integer $l$. We extend this to sequences with a shift, $(\fl{\alpha k+\beta})$.

\begin{lemma}\label{lem:linear}
Let $a,b$ be integers with $b>0$ and $\gcd(a,b)=1$, let $\beta\in\R$ and $m\ge2$. Then $(\fl{ak/b+\beta})_{k\ge1}$ is u.d.\ mod $m$ if and only if $\gcd(a,m)=1$.
\end{lemma}

\begin{proof}
Note that
\begin{equation}\label{eq:star}
\fl{a(k+b)/b+\beta}=\fl{ak/b+\beta}+a\qquad\text{for every }k.
\end{equation}

Suppose the sequence is u.d.\ mod $m$, and let $d=\gcd(a,m)$. Assume $d>1$. By \cite[Thm.~5.1]{Niv} the sequence is u.d.\ mod $d$. By~\eqref{eq:star} and $d\mid a$, it is periodic mod $d$ with period $b$. If $t$ is the number of $k\in\{1,\dots,b\}$ with $\fl{ak/b+\beta}\equiv0\pmod d$, then the frequency of the class $0$ is $t/b$. Uniform distribution forces $t/b=1/d$, so $d\mid b$. Since $d\mid a$ as well, $d\mid\gcd(a,b)=1$, a contradiction.

Conversely, suppose $\gcd(a,m)=1$. By~\eqref{eq:star} the sequence is periodic mod $m$ with period $bm$. For each residue $i$ modulo $m$ let $T_i$ be the set of $k\in\{1,\dots,bm\}$ with $\fl{ak/b+\beta}\equiv i\pmod m$. By~\eqref{eq:star}, the map $k\mapsto k+b$, taken modulo $bm$, sends $T_i$ injectively into $T_{i+a}$, with the index read modulo $m$. Hence $|T_i|\le|T_{i+a}|\le|T_{i+2a}|\le\cdots$. Since $a$ is invertible modulo $m$, the indices $i,i+a,i+2a,\dots$ run through all residues modulo $m$ and return to $i$. So all $|T_i|$ are equal, and since they sum to $bm$, each equals $b$. For $N=Qbm+R$ with $0\le R<bm$, the number of $k\le N$ in class $i$ lies between $Qb$ and $(Q+1)b$, which gives~\eqref{eq:ud}.
\end{proof}

\begin{theorem}\label{thm:linear-ud}
Let $\alpha,\beta\in\R$. Then $(\fl{\alpha k+\beta})_{k\ge1}$ is u.d.\ in $\Z$ if and only if $\alpha$ is irrational or $\alpha=1/l$ for some nonzero integer $l$.
\end{theorem}

\begin{proof}
If $\alpha$ is irrational, the sequence is u.d.\ in $\Z$ by Example~\ref{ex:weyl} (this case is also \cite[Thm.~3.2]{Niv}). If $\alpha=0$, the sequence is constant and is not u.d.\ mod $2$. Otherwise write $\alpha=a/b$ with $\gcd(a,b)=1$, $b>0$ and $a\ne0$. If $|a|=1$, then $\alpha=1/l$ with $l=\pm b$, and Lemma~\ref{lem:linear} gives u.d.\ mod every $m$. If $|a|\ge2$, Lemma~\ref{lem:linear} with $m=|a|$ shows that the sequence is not u.d.\ mod $|a|$.
\end{proof}

\subsection{The nonlinear case}

Niven showed that a nonlinear polynomial with integer coefficients is not u.d.\ in $\Z$ (Example~\ref{ex:niven}). We extend this to polynomials that are rational up to the constant term. We give two proofs. The first is elementary and uses moduli $p^2$. The second uses the Chebotarev density theorem and gives prime moduli, in a set of positive density.

\begin{theorem}\label{thm:nonlin-elem}
Let $P$ be a real polynomial of degree $n\ge2$ that is rational up to its constant term. Then there are infinitely many primes $p$ such that $\lambda(P)$ is not u.d.\ mod $p^2$. In particular, $\lambda(P)$ is not u.d.\ modulo infinitely many pairwise coprime integers.
\end{theorem}

\begin{proof}
Let $N$ and $Q$ be as in Section~\ref{sec:prelim}. The derivative $Q'$ is a nonconstant integer polynomial. By Schur's theorem~\cite{Sch}, infinitely many primes divide at least one nonzero value of $Q'$ on $\Z$. Fix such a prime $p>2N$ and an integer $a$ with $p\mid Q'(a)$.

Taylor expansion gives $Q(a+h)=Q(a)+hQ'(a)+h^2R(a,h)$ with $R\in\Z[x,h]$. Hence $Q(a+up)\equiv Q(a)\pmod{p^2}$ for every integer $u$. So for each $x\equiv a\pmod p$ we can write $Q(x)=Q(a)+e(x)p^2$ with $e(x)\in\Z$, and
\[
\fl{P(x)-P(0)}=\fl{\frac{Q(x)}{N}}=\fl{\frac{Q(a)+e(x)\,p^2}{N}}.
\]
Replacing $e$ by $e+N$ adds exactly $p^2$ to the right side. So, modulo $p^2$, the value of $\fl{P(x)-P(0)}$ depends only on $e(x)\bmod N$, and takes at most $N$ values. Since $\fl{P(x)}=\fl{P(x)-P(0)}+\fl{P(0)}+u$ with $u\in\{0,1\}$, the residues of $\fl{P(x)}$ modulo $p^2$, over all $x\equiv a\pmod p$, lie in a fixed set $\mathcal C$ of at most $2N$ classes.

Suppose $\lambda(P)$ were u.d.\ mod $p^2$. Among $1\le x\le sp^2$, the number of $x$ with $\fl{P(x)}\in\mathcal C$ would then be at most $(2N/p^2)\cdot sp^2+o(s)=2Ns+o(s)$. But this set contains all $sp$ integers $x\le sp^2$ with $x\equiv a\pmod p$. Letting $s\to\infty$ gives $p\le2N$, a contradiction. The moduli $p^2$ for distinct primes $p$ are pairwise coprime.
\end{proof}

\begin{theorem}\label{thm:nonlin-cheb}
Let $P$ be a real polynomial of degree $n\ge2$ that is rational up to its constant term. Then there is a set of primes $\mathcal P$ of positive density such that, for every $p\in\mathcal P$, $\lambda(P)$ is not u.d.\ mod $p$. More precisely, for $p\in\mathcal P$ some residue class modulo $p$ occurs in $\lambda(P)$ with frequency $n/p$.
\end{theorem}

\begin{proof}
Let $N$, $Q$ and $A_j$ be as in Section~\ref{sec:prelim}. Let $B_1,\dots,B_r$ be the distinct values among $A_0,\dots,A_{N-1}$, with $B_s$ occurring $\mu_s$ times, so $\mu_1+\dots+\mu_r=N$. By Lemma~\ref{lem:count}, for a prime $p\nmid N$ the count in~\eqref{eq:count} equals
\begin{equation}\label{eq:mult}
\sum_{s=1}^r\mu_s\cdot\#\{\,y\in\F_p: Q(y)+B_s-c\equiv0\pmod p\,\}.
\end{equation}

Choose an integer $c$ such that $F(x)=\prod_s\bigl(Q(x)+B_s-c\bigr)$ is squarefree. This holds for all but finitely many $c$. Indeed, $Q(x)+B_s-c$ has a repeated root only if $c-B_s$ is a value of $Q$ at a root of $Q'$, which excludes finitely many $c$. And two factors with $s\ne s'$ have no common root, because their difference $B_s-B_{s'}$ is a nonzero constant.

By Lemma~\ref{lem:cheb}, $F$ splits into $rn$ distinct linear factors modulo $p$ for a set of primes of positive density. Discarding the finitely many $p$ that divide $N$, we obtain the set $\mathcal P$. For $p\in\mathcal P$, each factor $Q(x)+B_s-c$ has exactly $n$ roots in $\F_p$, so~\eqref{eq:mult} equals $nN>N$. By Lemma~\ref{lem:count}, $\lambda(P)$ is not u.d.\ mod $p$. The residue $r$ with $Nr\equiv c$ occurs $nN$ times in a period of length $Np$, which is frequency $n/p$.
\end{proof}

\begin{remark}
Theorem~\ref{thm:nonlin-cheb} strengthens Theorem~\ref{thm:nonlin-elem} and also implies the statement ``if $\lambda(P)$ is u.d.\ mod all sufficiently large primes, then $P$ has an irrational non-constant coefficient''. The hypothesis $n\ge2$ is essential: $x/2$ is rational and $\lambda(x/2)$ is u.d.\ in $\Z$ by Theorem~\ref{thm:linear-ud}.
\end{remark}

\subsection{Proof of Theorem~\ref{thm:main-ud}}

If $P$ is constant, $\lambda(P)$ is constant and not u.d. If $P$ is linear, Theorem~\ref{thm:linear-ud} gives the claim. If $n\ge2$, sufficiency is Example~\ref{ex:weyl} and necessity is Theorem~\ref{thm:nonlin-elem} (or Theorem~\ref{thm:nonlin-cheb}). \qed

\section{Complete polynomials}

If $\lambda(P)$ is u.d.\ mod $m$, then $P$ is complete modulo $m$. For linear polynomials the converse fails, as the next theorem shows when compared with Theorem~\ref{thm:linear-ud}. For example, $P(x)=2x/3$ is complete in $\Z$ but $\lambda(P)$ is not u.d.\ mod $2$.

\subsection{Linear polynomials}

\begin{theorem}\label{thm:linear-complete}
The polynomial $P(x)=\alpha x+\beta$ is complete in $\Z$ if and only if $\alpha$ is irrational or $0<|\alpha|\le1$.
\end{theorem}

\begin{proof}
If $\alpha$ is irrational, apply Example~\ref{ex:weyl}. If $\alpha=0$, $\fl{P(x)}$ is constant, and $P$ is not complete modulo $2$.

Suppose $0<\alpha\le1$. Then $\fl{\alpha(x+1)+\beta}-\fl{\alpha x+\beta}\in\{0,1\}$ for every $x\in\Z$, and $\fl{\alpha x+\beta}\to\pm\infty$ as $x\to\pm\infty$. Hence $x\mapsto\fl{\alpha x+\beta}$ maps $\Z$ onto $\Z$, and $P$ is complete modulo every $m$. The case $-1\le\alpha<0$ is the same, with steps in $\{0,-1\}$.

Finally, let $\alpha=a/b$ with $\gcd(a,b)=1$, $b>0$ and $|a|>b$. For $x=bq+l$ with $0\le l<b$ we have
\[
\fl{\tfrac ab(bq+l)+\beta}=aq+\fl{\tfrac{al}{b}+\beta}.
\]
So modulo $|a|$, $\fl{P(x)}$ takes at most $b<|a|$ values, and $P$ is not complete modulo $|a|$.
\end{proof}

\subsection{Integer coefficients}

When $P(x)-P(0)$ has integer coefficients, completeness modulo $p$ is a question about permutation polynomials. A classical consequence of Hermite's criterion settles it in every degree.

\begin{proposition}\label{prop:integer}
Let $P$ be a real polynomial of degree $n\ge2$ with $P(x)-P(0)\in\Z[x]$. Then $P$ is not complete modulo any prime $p\equiv1\pmod n$ with $p\nmid a_n$. By Dirichlet's theorem there are infinitely many such primes, so $P$ is not complete in $\Z$.
\end{proposition}

\begin{proof}
Here $\fl{P(x)}=Q(x)+\fl{P(0)}$ with $Q=P-P(0)\in\Z[x]$. Modulo $p$, $Q$ reduces to a polynomial of degree $n$, with $1<n<p$ and $n\mid p-1$. By Hermite's criterion, no polynomial of such a degree is a permutation polynomial of $\F_p$ \cite[Thm.~7.4 and Cor.~7.5]{LN}. A map from the finite set $\F_p$ to itself that is not injective is not surjective. So some residue is missed by $Q$, hence by $\fl{P(x)}$.
\end{proof}

This is a quantitative form of the fact that a nonlinear integer polynomial cannot permute $\F_p$ for all large $p$. The full classification of integer polynomials that permute $\F_p$ for infinitely many $p$ is Schur's conjecture, proved by Fried~\cite{Fri}; see also~\cite{Tur}.

\subsection{Even degree}

When $P$ has non-integral rational coefficients, $\fl{P(x)}$ modulo $p$ is the union of $N$ shifted copies of the value set of $Q$ (Lemma~\ref{lem:periodic}(c)). The difficulty is to find a residue that all $N$ shifted copies miss at the same time. For even degree, complex conjugation supplies one.

\begin{lemma}\label{lem:noreal}
Let $F\in\Z[x]$ be a nonconstant polynomial with no real roots. Then the set of primes $p$ modulo which $F$ has no root has positive density.
\end{lemma}

\begin{proof}
Let $L$ be the splitting field of $F$ and $\sigma$ the restriction of complex conjugation to $L$. Then $\sigma\in\Gal(L/\Q)$, and $\sigma$ fixes no root of $F$, since $F$ has no real roots. Apply Lemma~\ref{lem:cheb}.
\end{proof}

Lemma~\ref{lem:noreal} complements the work of Berend and Bilu~\cite{BB} and of Brandl, Bubboloni and Hupp~\cite{BBH} on polynomials that have a root modulo every prime or every integer. Such polynomials need not have a rational root, but by Lemma~\ref{lem:noreal} they must have a real one.

\begin{lemma}\label{lem:even}
Let $Q\in\Z[x]$ have even degree $n\ge2$ and let $B_1,\dots,B_r$ be integers. Then there are an integer $c$ and a set of primes of positive density such that, for each such $p$, $Q(y)+B_s\not\equiv c\pmod p$ for all $y\in\F_p$ and all $s$.
\end{lemma}

\begin{proof}
Replacing $Q$, $B_s$ and $c$ by their negatives if necessary, we may assume that the leading coefficient of $Q$ is positive. Then $Q$ is bounded below on $\R$, say $Q(x)\ge-K$. Choose an integer $c<\min_sB_s-K$. Then each $Q(x)+B_s-c$ is positive on $\R$, so $F(x)=\prod_s\bigl(Q(x)+B_s-c\bigr)$ has no real roots. By Lemma~\ref{lem:noreal}, $F$ has no root modulo $p$ for a positive-density set of primes $p$, which is the claim.
\end{proof}

\subsection{Proof of Theorem~\ref{thm:even}}

(2)$\Rightarrow$(3) is Example~\ref{ex:weyl}. (3)$\Rightarrow$(4) holds because u.d.\ mod $m$ implies completeness modulo $m$. (4)$\Rightarrow$(1) is trivial.

(1)$\Rightarrow$(2): suppose $P$ has even degree $n\ge2$ and is rational up to its constant term. Let $N$, $Q$ and $A_0,\dots,A_{N-1}$ be as in Section~\ref{sec:prelim}. $Q$ has even degree $n$. Apply Lemma~\ref{lem:even} to $Q$ and the integers $A_j$. We obtain an integer $c$ and a positive-density set of primes $p$, which we may take with $p\nmid N$, such that $c\notin V_p$. By Lemma~\ref{lem:periodic}(c), $P$ is not complete modulo any of these primes, so (1) fails. \qed

\medskip
The proof shows more: if $P$ has even degree and is rational up to its constant term, then $P$ fails to be complete modulo every prime in a set of positive density.

\section{Complete monomials}

For odd degree, complex conjugation no longer works, because a real polynomial of odd degree has a real root. For $P(x)=ax^n+c$ we can replace complex conjugation by a Kummer automorphism. The relevant value sets are cosets of the subgroup of $n$-th powers in $\F_p^\times$, so residuosity questions enter. Classical work on runs of power residues and non-residues goes back to Brauer~\cite{Bra}; see also~\cite{RJ} and~\cite{Mil}.

\begin{lemma}\label{lem:kummer}
Let $q$ be a prime and let $c_1,\dots,c_r$ be nonzero integers. Suppose there are distinct primes $\ell_1,\dots,\ell_r$, all different from $q$, such that the $\ell_s$-adic valuation of $c_t$ is $1$ if $s=t$ and $0$ if $s\ne t$. Then the set of primes $p\equiv1\pmod q$ for which none of $c_1,\dots,c_r$ is a $q$-th power modulo $p$ has positive density.
\end{lemma}

\begin{proof}
Let $\zeta$ be a primitive $q$-th root of unity, $K=\Q(\zeta)$, and $L=K(c_1^{1/q},\dots,c_r^{1/q})$. $L$ is the splitting field of $(x^q-1)\prod_s(x^q-c_s)$, so it is Galois over $\Q$.

We first show that the classes of $c_1,\dots,c_r$ in $K^\times/(K^\times)^q$ are independent over $\F_q$. Suppose $\prod_tc_t^{e_t}=\gamma^q$ with $\gamma\in K$ and $0\le e_t<q$, not all $e_t=0$. Choose $s$ with $e_s\ne0$. Since $\ell_s\ne q$, the prime $\ell_s$ is unramified in $K$. So for a prime $\mathfrak l$ of $K$ above $\ell_s$, the valuation $v_{\mathfrak l}$ restricts to the $\ell_s$-adic valuation on $\Q$. Then $v_{\mathfrak l}\bigl(\prod_tc_t^{e_t}\bigr)=e_s$, which is not divisible by $q$, whereas $v_{\mathfrak l}(\gamma^q)=q\,v_{\mathfrak l}(\gamma)$. This is a contradiction.

By Kummer theory \cite[Ch.~VI, \S8]{Lan}, the map $\tau\mapsto\bigl(\tau(c_s^{1/q})/c_s^{1/q}\bigr)_{s=1,\dots,r}$ is an isomorphism from $\Gal(L/K)$ onto $\mu_q^r$. Hence there is $\sigma\in\Gal(L/K)$ with $\sigma(c_s^{1/q})=\zeta\,c_s^{1/q}$ for every $s$. In particular $\sigma$ fixes no root of any $x^q-c_s$.

By Lemma~\ref{lem:cheb}(b), the primes $p$ whose Frobenius class contains $\sigma$ have positive density. Discard the finitely many $p$ dividing $q\,c_1\cdots c_r$. For each remaining $p$, the Frobenius is trivial on $K$, so $p$ splits completely in $K$, that is, $p\equiv1\pmod q$. Suppose some $c_s$ were a $q$-th power modulo $p$, say $c_s\equiv y^q$. Then $y\not\equiv0$ and $\F_p$ contains the $q$-th roots of unity, so $x^q-c_s$ has $q$ distinct roots in $\F_p$. By Lemma~\ref{lem:cheb}(a), applied to $x^q-c_s$, the Frobenius would fix all $q$ roots of $x^q-c_s$, but the conjugates of $\sigma$ fix none. Hence no $c_s$ is a $q$-th power modulo $p$.
\end{proof}

As a by-product we obtain a correct version of the statement on runs of non-residues that the original argument relied on.

\begin{corollary}\label{cor:runs}
Let $n\ge2$ and $l\ge1$. There are infinitely many primes $p$ \textup{(}indeed, a set of positive density\textup{)} together with a positive integer $t$ such that $t,t+1,\dots,t+l-1$ are all $n$-th power non-residues modulo $p$.
\end{corollary}

\begin{proof}
Let $q$ be a prime divisor of $n$. An $n$-th power is a $q$-th power, so it suffices to treat $q$. Choose distinct primes $\ell_0,\dots,\ell_{l-1}>\max(l,q)$. By the Chinese remainder theorem, choose $t>0$ with $t+i\equiv\ell_i\pmod{\ell_i^2}$ for each $i$. Then $\ell_i$ exactly divides $t+i$. For $k\ne i$ we have $t+k\equiv k-i\not\equiv0\pmod{\ell_i}$, since $0<|k-i|<\ell_i$. Apply Lemma~\ref{lem:kummer} with $c_i=t+i$.
\end{proof}

Note that the integer $t$ is fixed before the primes $p$ are produced.

\begin{theorem}\label{thm:monomial}
Let $P(x)=ax^n+c$ with $n\ge2$, $a\in\Q$ and $c\in\R$. Then $P$ fails to be complete modulo every prime in a set of positive density. In particular, $P$ is not complete in $\Z$ and $\lambda(P)$ is not u.d.\ in $\Z$. Together with Example~\ref{ex:weyl}: $ax^n+c$ is complete modulo all sufficiently large primes if and only if $a\notin\Q$.
\end{theorem}

\begin{proof}
If $a=0$ there is nothing to prove. Write $a=M/N$ with $\gcd(M,N)=1$, $N>0$ and $M\ne0$. In the notation of Section~\ref{sec:prelim}, this $N$ is the same, $Q(x)=Mx^n$, and $A_j=N\fl{P(j)}-Mj^n$. Let $q$ be a prime divisor of $n$, and let $B_1,\dots,B_r$ be the distinct values among $A_0,\dots,A_{N-1}$.

Choose distinct primes $\ell_1,\dots,\ell_r$ that do not divide $qM$ and exceed $|B_s-B_t|$ for all $s\ne t$. By the Chinese remainder theorem, choose an integer $c$ with $c\equiv B_s+\ell_s\pmod{\ell_s^2}$ for every $s$. Then $\ell_s$ exactly divides $c-B_s$. For $t\ne s$, $c-B_t\equiv B_s-B_t\not\equiv0\pmod{\ell_s}$. Put $c_s=M^{q-1}(c-B_s)$. Since $\ell_s\nmid M$, the integers $c_s$ and primes $\ell_s$ satisfy the hypothesis of Lemma~\ref{lem:kummer}.

Lemma~\ref{lem:kummer} gives a positive-density set of primes $p$, which we may take with $p\nmid NM$, such that no $c_s$ is a $q$-th power modulo $p$. Suppose $My^n+B_s\equiv c\pmod p$ for some $y$ and $s$. Multiplying by $M^{q-1}$ gives
\[
\bigl(M\,y^{n/q}\bigr)^q\equiv M^{q-1}(c-B_s)=c_s\pmod p,
\]
so $c_s$ would be a $q$-th power modulo $p$. Hence $c\notin V_p$, and by Lemma~\ref{lem:periodic}(c), $P$ is not complete modulo $p$.
\end{proof}

For even $n$, Theorem~\ref{thm:monomial} is contained in Theorem~\ref{thm:even}. For odd $n$ it is new: the even-degree argument cannot apply, since every odd-degree real polynomial has a real root.

\section{Concluding remarks and an open problem}\label{sec:open}

For uniform distribution the picture is complete (Theorem~\ref{thm:main-ud}). For completeness, one case is open: polynomials of odd degree $n\ge3$ that are rational up to their constant term, have some non-integral coefficient, and are not of the form $ax^n+c$. By Lemma~\ref{lem:periodic}(c), Theorem~\ref{thm:even} extends to all nonlinear polynomials if the following question has a positive answer.

\begin{problem}\label{prob:open}
Let $Q\in\Z[x]$ have odd degree $n\ge3$, and let $B_1,\dots,B_r$ be integers. Is there an integer $c$ such that, for infinitely many primes $p$, $Q(y)+B_s\not\equiv c\pmod p$ for all $y\in\F_p$ and all $s$?
\end{problem}

A sufficient condition is that the value set of $Q$ modulo $p$ misses $l$ consecutive residues, with $l=\max_sB_s-\min_sB_s+1$: if $Q$ misses $k,\dots,k+l-1$ modulo $p$, take $c=k+\max_sB_s$.

By Lemma~\ref{lem:cheb}, Problem~\ref{prob:open} has a positive answer whenever, for some $c$, the Galois group of $F_c(x)=\prod_s\bigl(Q(x)+B_s-c\bigr)$ contains an element that fixes no root of $F_c$. Three cases are settled.
\begin{itemize}
\item $r=1$: Proposition~\ref{prop:integer}, via Hermite's criterion. Alternatively, choose $c$ with $Q-c$ irreducible (Hilbert's irreducibility theorem), and use Jordan's theorem that a transitive permutation group of degree $\ge2$ contains a fixed-point-free element~\cite{Ser}.
\item Even degree: complex conjugation (Lemma~\ref{lem:even}).
\item $Q=Mx^n$: a Kummer automorphism that is trivial on $\Q(\zeta_q)$ (Theorem~\ref{thm:monomial}).
\end{itemize}
The difficulty in general is that a single $\sigma$ must act without fixed points on the roots of every factor at once. If the splitting fields of the factors $Q(x)+B_s-c$ were linearly disjoint over $\Q$, Jordan's theorem applied factor by factor would suffice. They need not be: for $Q=x^n$ they all contain $\Q(\zeta_n)$, which is why Section~5 works over that field.

A second direction concerns generalized polynomials, built from ordinary polynomials using the floor function together with addition and multiplication; an example is $\fl{\fl{a_1x^2+a_2}x}+\fl{a_3x+a_4}x^2$. H{\aa}land~\cite{Hal} gave conditions, in terms of the rational independence of the coefficients, under which such sequences are u.d.\ mod~$1$. Bergelson and Leibman~\cite{BL} described the distribution of bounded generalized polynomials. The sequence $\fl{P(k)}$ is the simplest generalized polynomial, and it would be natural to study uniform distribution and completeness modulo integers for broader classes.

\section*{Acknowledgement}
I would like to thank Siena College for providing a sabbatical leave during which this work was completed. I would also like to thank the referee for many useful comments and suggestions. This revised version was completed by an AI assistant (Claude).


\begin{thebibliography}{99}

\bibitem{BB} D. Berend and Y. Bilu, Polynomials with roots modulo every integer, \emph{Proc. Amer. Math. Soc.} \textbf{124} (1996), 1663--1671.

\bibitem{BBH} R. Brandl, D. Bubboloni and I. Hupp, Polynomials with roots mod $p$ for all primes $p$, \emph{J. Group Theory} \textbf{4} (2001), 233--239.

\bibitem{BL} V. Bergelson and A. Leibman, Distribution of values of bounded generalized polynomials, \emph{Acta Math.} \textbf{198} (2007), 155--230.

\bibitem{Bra} A. Brauer, \"Uber Sequenzen von Potenzresten, \emph{Sitzungsber. Preuss. Akad. Wiss. Phys.-Math. Kl.} (1928), 9--16.

\bibitem{BSD} B. J. Birch and H. P. F. Swinnerton-Dyer, Note on a problem of Chowla, \emph{Acta Arith.} \textbf{5} (1959), 417--423.

\bibitem{Fri} M. Fried, On a conjecture of Schur, \emph{Michigan Math. J.} \textbf{17} (1970), 41--55.

\bibitem{Hal} I. J. H{\aa}land, Uniform distribution of generalized polynomials, \emph{J. Number Theory} \textbf{45} (1993), 327--366.

\bibitem{HL} G. H. Hardy and J. E. Littlewood, Some problems of Diophantine approximation: Part I. The fractional part of $n^k\theta$, \emph{Acta Math.} \textbf{37} (1914), 155--191.

\bibitem{KN} L. Kuipers and H. Niederreiter, \emph{Uniform Distribution of Sequences}, Wiley, 1974.

\bibitem{Kok} J. F. Koksma, Ein mengentheoretischer Satz \"uber die Gleichverteilung modulo Eins, \emph{Compositio Math.} \textbf{2} (1935), 250--258.

\bibitem{Lan} S. Lang, \emph{Algebra}, revised 3rd ed., Graduate Texts in Mathematics 211, Springer, 2002.

\bibitem{LeV} W. J. LeVeque, Note on a theorem of Koksma, \emph{Proc. Amer. Math. Soc.} \textbf{1} (1950), 380--383.

\bibitem{LN} R. Lidl and H. Niederreiter, \emph{Finite Fields}, 2nd ed., Encyclopedia of Mathematics and its Applications 20, Cambridge University Press, 1997.

\bibitem{LS} H. W. Lenstra, Jr. and P. Stevenhagen, Chebotar\"ev and his density theorem, \emph{Math. Intelligencer} \textbf{18} (1996), 26--37.

\bibitem{Mil} W. H. Mills, Characters with preassigned values, \emph{Canad. J. Math.} \textbf{15} (1963), 169--171.

\bibitem{Nar} W. Narkiewicz, \emph{Uniform Distribution of Sequences of Integers in Residue Classes}, Lecture Notes in Mathematics 1087, Springer, 1984.

\bibitem{Neu} J. Neukirch, \emph{Algebraic Number Theory}, Grundlehren der mathematischen Wissenschaften 322, Springer, 1999.

\bibitem{Niv} I. Niven, Uniform distribution of sequences of integers, \emph{Trans. Amer. Math. Soc.} \textbf{98} (1961), 52--61.

\bibitem{RJ} J. R. Rabung and J. H. Jordan, Consecutive power residues or nonresidues, \emph{Math. Comp.} \textbf{24} (1970), 737--740.

\bibitem{Sch} I. Schur, \"Uber die Existenz unendlich vieler Primzahlen in einigen speziellen arithmetischen Progressionen, \emph{Sitzungsber. Berl. Math. Ges.} \textbf{11} (1912), 40--50.

\bibitem{Ser} J.-P. Serre, On a theorem of Jordan, \emph{Bull. Amer. Math. Soc. (N.S.)} \textbf{40} (2003), 429--440.

\bibitem{Tur} G. Turnwald, On Schur's conjecture, \emph{J. Austral. Math. Soc. Ser. A} \textbf{58} (1995), 312--357.

\bibitem{Uch} S. Uchiyama, On the uniform distribution of sequences of integers, \emph{Proc. Japan Acad.} \textbf{37} (1961), 605--609.

\bibitem{Wey} H. Weyl, \"Uber die Gleichverteilung von Zahlen mod. Eins, \emph{Math. Ann.} \textbf{77} (1916), 313--352.

\end{thebibliography}
\end{document}